\documentclass[12pt]{amsart}
 \usepackage{amsthm,amsmath,amssymb,amscd,epsfig,mathtools,tabu}

 {\end{list}}
 {
    \newtheorem{theorem}{Theorem}[section]
    
    \newtheorem{lemma}[theorem]{Lemma}

    \newtheorem{conjecture}[theorem]{Conjecture}

 }
 {\theoremstyle{definition}

    \newtheorem{example}[theorem]{Example}
    \newtheorem{definition}[theorem]{Definition}
    \newtheorem{remark}[theorem]{Remark}
 }

 \newcommand{\bc}{{\bf c}}
 \newcommand{\bd}{{\bf d}}

 \newcommand{\cv}{{\bf c}}
 \newcommand{\dv}{{\bf d}}
 
 \newcommand{\cd}{\cv\dv}

 \newcommand{\RR}{{\mathbb{R}}}

 \newcommand{\PP}{{\mathbb{P}}}
 \newcommand{\ZZ}{{\mathbb{Z}}}

  \newcommand{\mdeg}{\operatorname{mdeg}}
  \renewcommand{\min}{\operatorname{min}}
  \newcommand{\supp}{\operatorname{supp}}
 \newcommand{\Int}{\operatorname{Int}}

 \newcommand{\setmin}{{\smallsetminus}}

\begin{document}
 \title{$M$-vector analogue for the $\cd$-index}

 \author{Kalle Karu}
 \address{Department of Mathematics\\ University of British Columbia \\
   1984 Mathematics Road\\
 Vancouver, B.C. Canada V6T 1Z2}
 \email{karu@math.ubc.ca}

 \begin{abstract}
A well-known conjecture of McMullen, proved by Billera, Lee and Stanley, describes the face numbers of simple polytopes. The necessary and sufficient condition is that the toric $g$-vector of the polytope is an $M$-vector, that is, the vector of dimensions of graded pieces of a standard graded algebra $A$. Recent work by Murai, Nevo  and Yanagawa suggests a similar condition for the coefficients of the $\cd$-index of a poset $P$. The coefficients of the $\cd$-index are conjectured to be the dimensions of graded pieces in a standard multigraded algebra $A$.  We prove the conjecture for simplicial spheres and we give numerical evidence for general shellable spheres. In the simplicial case we construct the multi-graded algebra $A$ explicitly using lattice paths. 
 \end{abstract}

 \maketitle

 \section{Introduction}
 \setcounter{equation}{0}

We start by recalling the conjecture of McMullen \cite{McMullen}, proved by Billera, Lee and Stanley \cite{BilleraLee, Stanley}, that describes all possible face numbers of simple polytopes. A polytope $P$ is simple if its normal fan is simplicial. Let $f_i$ be the number of  $i$-dimensional cones in the normal fan, $i=0,\ldots,n=\dim P$. One encodes these numbers in the toric $h$-vector $(h_0,\ldots,h_n)$, using the equality of polynomials
\[ \sum_{i=0}^n h_i t^i = \sum_{i=0}^n f_i(t-1)^{n-i}.\]
The $h$-numbers satisfy the Euler's equation $h_0=1$ and the Dehn-Sommerville equations $h_i = h_{n-i}$. Define the toric $g$-vector $(g_0,\ldots,g_{\lfloor n/2 \rfloor})$ by
\[ g_0=1,\quad g_i = h_i-h_{i-1}, \text{ for } i>0.\]
Then all linear relations among the face numbers $f_i$ are given by $g_0=1$, $g_i\geq 0$. However, there are additional non-linear relations. The precise condition is that there exists a standard graded $k$-algebra $A = \oplus_i A_i$, with $A_0=k$ the ground field and $A$ generated in degree one, such that 
\[ g_i = \dim A_i, \quad \text{ for } i \geq 0.\]
Such vectors $(g_i)$ that are given by dimensions of graded pieces of a standard graded algebra $A$ are called M-vectors (after Macaualy) and they can be described by nonlinear inequalities. 

Murai and Nevo \cite{MuraiNevo} conjectured that similar nonlinear relations hold among coefficients of the $\cd$-index of a poset $P$. Further  evidence for this conjecture was given by Murai and Yanagawa \cite{MuraiYanagawa}. The conjecture is that there exists a standard multigraded algebra $A$, such that the dimensions of the graded pieces of $A$ give the coefficients of the $\cd$-index. 

Let us recall the definition of the $\cd$-index so that we can state the precise conjecture. Let $P$ be a graded poset of rank $n+1$, with rank function $\rho$. We consider chains in $P$
\[ \hat{0} < x_1 < x_2 < \cdots < x_m < \hat{1}. \]
The type of this chain is $S = \{\rho(x_1),\ldots,\rho(x_m)\} \subset \{1,\ldots,n\}$. Let $f_S$ be the number of chains of type $S$ in $P$. We encode these flag numbers in the numbers $h_S$ by the formula
\[ \Psi_P(t_1,\ldots,t_n) := \sum_{S} h_S t^S = \sum_S f_S (t-1)^{\overline{S}},\]
where 
\[ t^S = \prod_{i\in S} t_i,\quad (t-1)^T = \prod_{i\in T} (t_i-1), \quad \overline{S} = \{1,\ldots,n\}\setmin S.\]
When the poset $P$ is Eulerian, then the numbers $h_S$ satisfy the generalized Dehn-Sommerville equations given by Bayer and Billera \cite{BayerBillera}, for example $h_S = h_{\overline{S}}$. The complete set of these linear relations is described by writing $\Psi_P$  as a polynomial in two noncommuting variables $\bc$ and $\bd$ of degree $1$ and $2$, respectively. A $\cd$-monomial corresponds to a polynomial in the variables $t_i$ by replacing $\bc$ with $t_i+1$ and $\bd$ with $t_i+t_{i+1}$ as illustrated in the example:
\[ \bc\bd\bd\bc = (t_1+1)(t_2+t_3)(t_4+t_5)(t_6+1).\]
Then by Bayer and Klapper \cite{BayerKlapper}, the polynomial $\Psi_P$ for an Eulerian poset $P$ of rank $n+1$ can be expressed as a homogeneous polynomial of degree $n$ in the variables $\bc$ and $\bd$. The polynomial $\Psi_P(\bc,\bd)$ is called the $\cd$-index of $P$. When $P$ is Gorenstein* (this means, the simplicial complex formed by chains in $P$ is a purely $(n-1)$-dimensional homology sphere), then the $\cd$-index has non-negative coefficients \cite{Stanley2, Karu}. Moreover, the $\cd$-index is the most efficient encoding of the flag numbers $f_S$ in the sense that there are no more linear relations among the coefficients of $\Psi_P(\bc,\bd)$ that hold for all Gorenstein* posets $P$, other than the coefficients being nonnegative and the coefficient of $\bc^n$ being equal to one \cite{Stanley2}.


Define the multidegree of a degree $n$ $\cd$-monomial $M$ as a zero-one vector in $\ZZ^n$ obtained by replacing each $\bc$ with $0$ and each $\bd$ with $10$. For example,
\[ \mdeg(\bc\bd\bd\bc) = (0,1,0,1,0,0).\]
For a $\cd$-monomial $M$ with multidegree $v=\mdeg(M)$, write $\Psi_{P,v}$ for the coefficient of $M$ in $\Psi_P(\bc,\bd)$. Let $\Psi_{P,v}=0$ for vectors $v$ that are not multidegrees of $\cd$-monomials.

We say that a $\ZZ^n$-graded $k$-algebra $A$ (associative, commutative, with $1$) is standard multigraded if $A_0 = k$ and $A$ is generated in degrees $e_i = (0,\ldots,0,1,0,\ldots 0)$ for $i=0,\ldots,n$. 

The following conjecture is the main subject of this article.

\begin{conjecture}[Murai-Nevo] \label{conj-main} Let $P$ be a Gorenstein* poset of rank $n+1$. Then there exists a standard $\ZZ^n$-graded $k$-algebra $A = \oplus_v A_v$, such that for any $v\in\ZZ^n$
\[ \Psi_{P,v} = \dim A_v. \] 
\end{conjecture}

In \cite{MuraiNevo} the conjecture was given in a different but equivalent form. Conjecture~4.3 in \cite{MuraiNevo} states that the numbers  $\Psi_{P,v}$ should be the flag numbers of an $(n-1)$-coloured simplicial complex. Stated differently, Murai and Nevo conjectured that the algebra $A$ in Conjecture~\ref{conj-main} should be a polynomial ring modulo monomial relations. An arbitrary multigraded algebra $A$ can be degenerated to this form using a Gr\"obner basis.

Murai and Yanagawa \cite{MuraiYanagawa} proved that for any Gorenstein* poset $P$ and any set $S=\{i_1,\ldots,i_m\}$,
\[ \Psi_{P, e_{i_1}+\cdots+e_{i_m}} \leq \prod_{j=1}^m \Psi_{P,e_{i_j}}.\]
The case $m=2$ of this result was proved earlier in \cite{MuraiNevo}. If Conjecture~\ref{conj-main} is true, then this inequality follows from the surjectivity of the multiplication map
\[ A_{e_{i_1}}\otimes A_{e_{i_2}}\otimes\cdots \otimes  A_{e_{i_m}} \to A_{e_{i_1}+\cdots+e_{i_m}}.\]

In the shellable case we generalize this result to:

\begin{theorem} \label{thm-ineq} Let $P$ be a shellable Gorenstein* poset of rank $n+1$. Then for any $v,w\in \ZZ^n_{\geq 0}$
\[ \Psi_{P,v+w} \leq \Psi_{P,v} \cdot \Psi_{P,w}.\]
\end{theorem}

The exact definition of shellability that we use here is given in Section~2 below. Shellable posets include all face posets of convex polytopes.

The inequalities given in Theorem~\ref{thm-ineq} unfortunately are not enough for the existence of an algebra $A$ as in Conjecture~\ref{conj-main}. For example, consider the degree $6$ homogeneous $\cd$-polynomial
\[ \bc^6 +2(\bd\bc^4+\bc^2\bd\bc^2+\bc^4\bd)+(\bc^2\bd^2+\bd\bc^2\bd+\bd^2\bc^2) +2\bd^3.\]
This polynomial satisfies all the inequalities of the theorem, but it is not hard to check that there cannot exist a standard multigraded $k$-algebra $A$ as in the conjecture. I do not know if this polynomial is the $\cd$-index of any Gorenstein* poset.

 The $\cd$-index of a simplicial complex $P$ is much easier to understand. In fact the $\cd$-index is determined by the face numbers $f_i$ of $P$. For simplicial complexes we prove:

\begin{theorem} \label{thm-simpl} Let $P$ be the poset of a Gorenstein* simplicial complex. Then Conjecture~\ref{conj-main} holds for $P$.
\end{theorem}

We give an explicit construction of the algebra $A$ in Section~4 below. 

In the special case where $P$ is the Boolean lattice $B_{n+1}$, the algebra $A$ can be described in terms of lattice paths as follows.

\begin{definition} \label{def-admfun}
 Let $M$ be a $\cd$-monomial of multidegree
\[ \mdeg(M) = v = (v_1,\ldots,v_n).\]
A function
\[ f: \{0,1,\ldots,n\} \to \ZZ \]
is called {\em admissible} for $M$ if the following is satisfied
\begin{itemize}
\item (Range of $f$) $f(0)=0$, $f(n)=n$,  $0\leq f(i) \leq i$ for $i=1,\ldots,n-1$.
\item (Strict ascent for $0$) If $v_i=0$, then $f(i-1)< f(i)$.
\item (Weak descent for $1$) If $v_i=1$, then $f(i-1)\geq f(i)$.
\item (Bound on descent) If $v_i=1$, $i>1$, then $f(i-1)-f(i)\leq f(i-2)+1$.
\end{itemize}
\end{definition}
 
Admissible functions give the coefficients of the $\cd$-index of $B_{n+1}$:

\begin{theorem} \label{thm-Bn}
 The coefficient of a $\cd$-monomial $M$ in $\Psi_{B_{n+1}}$ is the number of admissible functions for $M$.
\end{theorem}

This theorem is due to Fan and He \cite{FanHe}. Our notation is a bit different from their's and we also give a different proof using a shelling of $B_{n+1}$. There are other combinatorial interpretations of the coefficients of  $\Psi_{B_{n+1}}$, for example in terms of Andr\'e permutations \cite{Purtill, Stanley2}.

\begin{figure} [ht]
\begin{minipage}[h]{0.45\linewidth}
\centerline{\psfig{figure=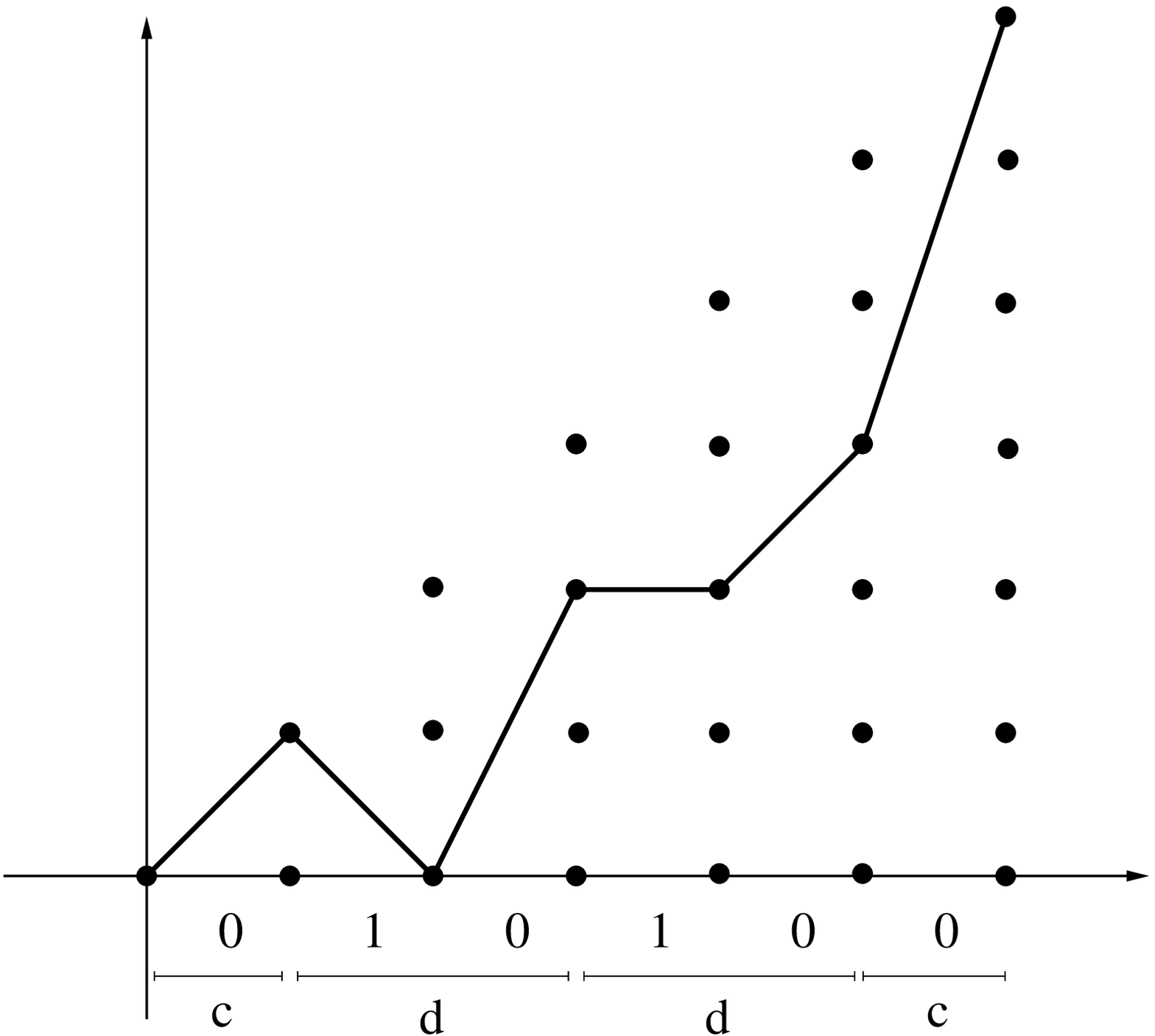,width=1\linewidth}}
\centerline{(a)}
\end{minipage}
\begin{minipage}[h]{0.45\linewidth}
\centerline{\psfig{figure=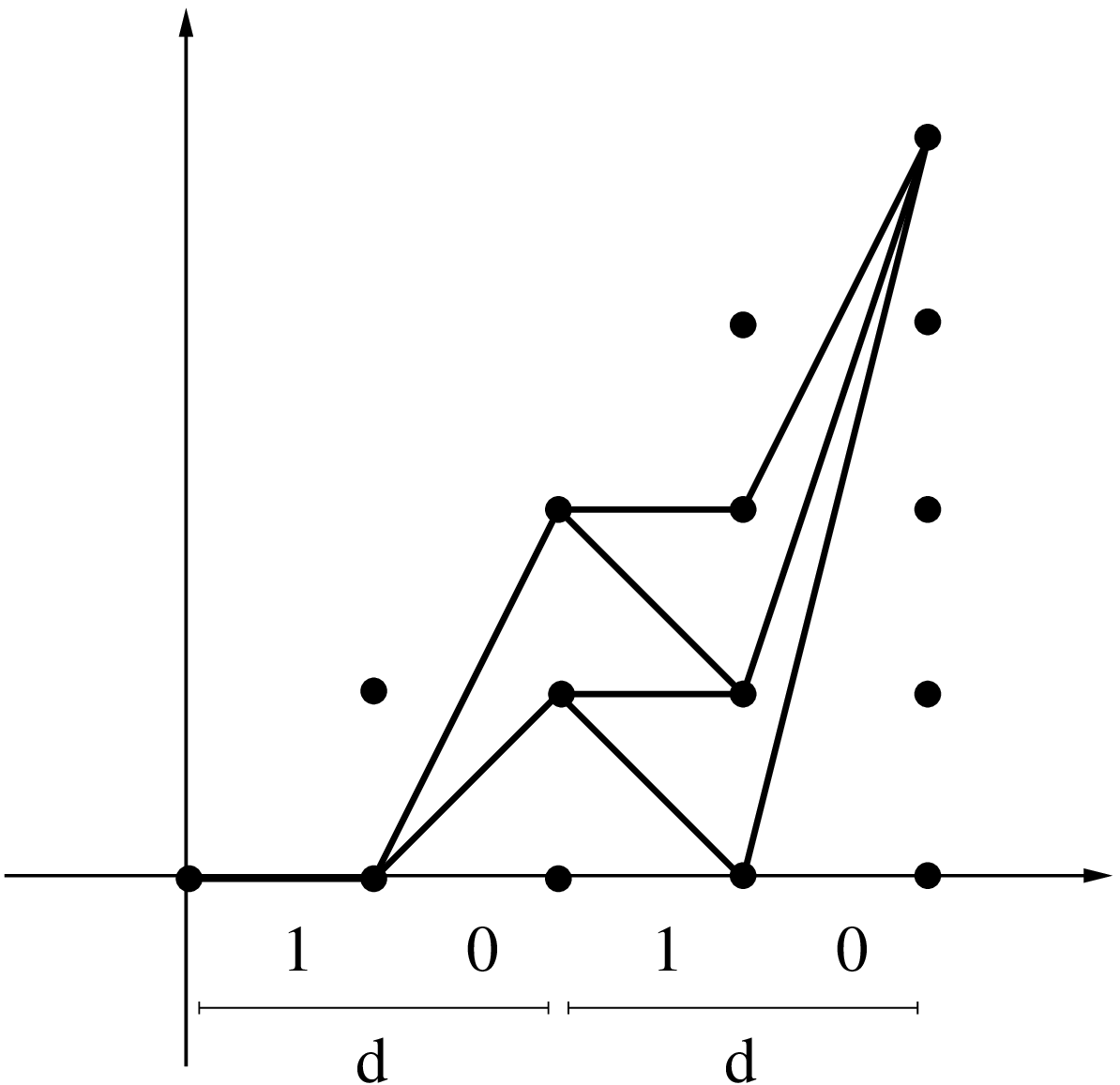,width=.95\linewidth}}
\centerline{(b)}
\end{minipage}
\caption{Admissible functions for $\bc\bd\bd\bc$ and $\bd\bd$.}
\end{figure}

To visualize admissible functions $f$, we extend them to piecewise linear functions $f: [0,n]\to\RR$. Figure~1(a) shows an example of an admissible function for $M=\bc\bd\bd\bc$. Figure~1(b) shows all $4$ admissible functions for $M=\bd\bd$. Note that the function $f=(0,0,2,0,4)$ is not admissible for $\bd\bd$ because it does not satisfy the ``bound on descent'' condition. The full $\cd$-index of $B_{5}$ is 
\[ \Psi_{B_5} = \bc^4 + 3 \bd \bc^2 + 5 \bc\bd\bc + 3\bc^2 \bd + 4\bd^2.\]

We define the algebra $A$ for $B_{n+1}$ by taking the admissible functions for all $\cd$-monomials of degree $n$ as the $k$-basis. The product of two admissible functions is defined to be either zero or an admissible function. More precisely, let $M_1$, $M_2$, $M_3$ be $\cd$-monomials of degree $n$, so that
\[ \mdeg(M_1) + \mdeg(M_2) = \mdeg(M_3).\]
Let $f_1$ and $f_2$ be two functions admissible for $M_1$ and $M_2$, respectively. Define
\[ f_1 \cdot f_2  = \begin{cases}
                    \min(f_1,f_2) & \text{if $\supp(f_1)\cap\supp(f_2) = \emptyset$ and $\min(f_1,f_2)$ is admissible for $M_3$,} \\
		    0 & \text{otherwise.}
                   \end{cases}
\]
Here $\min(f_1,f_2)$ is computed pointwise and the support of $f$ is
\[ \supp(f) = \{i\in\{0,\ldots,n\}| f(i)<i\}.\]
We prove in Section~3 that this multiplication turns $A$ into a standard multigraded $k$-algebra, hence $B_{n+1}$ satisfies Conjecture~\ref{conj-main}. The algebra $A$ for a general Gorenstein* simplicial complex  is constructed similarly using lattice paths.

It is not difficult to see that an admissible function $f$ in the algebra $A$ constructed above is a product of a unique set of admissible functions $f_{i_j}$ of multidegree $e_{i_j}$. This means that the algebra $A$ is monomial and one can construct from it an $n$-coloured simplicial complex as conjectured by Murai and Nevo. Each admissible function of degree $e_i$ gives a vertex coloured with $i$ and a set of these vertices span a simplex if the product of the corresponding functions is nonzero.

{\bf Acknowledgement.} I thank Satoshi Murai for interesting conversation about the main conjecture and for explaining the equivalence of Conjecture~\ref{conj-main} and Conjecture~4.3 in \cite{MuraiNevo}.

\section{Shellable fans}

We will use the terminology of fans as in the theory of toric varieties \cite{Fulton}. Since we are only interested in the poset of cones in a fan $\Delta$, everything below also applies to the posets of regular $CW$-complexes. Thus, our fans are not fans over polyhedral complexes, but fans over $CW$-complexes. This implies for example that the intersection of two cones in a fan $\Delta$ is not necessarily a cone but a subfan of the boundary of each cone. One can generalize this further to Gorenstein* posets, but since we are interested in shellable posets, we do not get anything beyond $CW$-complexes.

\begin{definition}
 An $n$-dimensional fan $\Delta$ is a fan over a regular $(n-1)$-dimensional $CW$-complex $C$, where $C$ is homeomorphic to the sphere $S^{n-1}$ or to the disk $D^{n-1}$.
\end{definition}

In particular, a fan is Gorenstein* if and only if it is a fan over a regular $CW$-sphere.

For an $n$-dimensional fan $\Delta$, let $\Delta^{\leq m}$ be its $m$-skeleton, consisting of all cones of dimension at most $m$. Also let $\partial\Delta$ be the boundary, which is a subfan of $\Delta^{\leq n-1}$ and let $\Int(\Delta) = \Delta\setmin \partial\Delta$ (as a set of cones). Similarly, let
\[ (\Int \Delta)^{\leq n-1} = \Delta^{\leq n-1}\setmin \partial\Delta.\]

We define shellability inductively on dimension.

\begin{definition}\label{def-shell}
The unique $0$-dimensional fan is shellable. For $n>0$, an $n$-dimensional fan $\Delta$ is shellable if $\partial\Delta$ is shellable and there exists a decomposition
\[ (\Int \Delta)^{\leq n-1} = \bigsqcup_i \Int \Delta_i,\]
where each $\Delta_i$ is an $n-1$ dimensional shellable fan (over a sphere or a disks).
\end{definition}

\begin{remark} The usual shelling of an $n$-dimensional fan consists of ordering maximal cones $\sigma_1,\ldots,\sigma_N$, so that if $\Sigma_j$ is the fan generated by $\sigma_1,\ldots,\sigma_j$, then $\sigma_{j+1}\cap \Sigma_j$ is a purely $(n-1)$-dimensional fan (over a sphere or a disk). Taking $\Delta_j = \sigma_{j+1}\cap \Sigma_j$ gives a shelling as in Definition~\ref{def-shell}, provided that $\partial \Delta$ and $\Delta_j$ are again shellable. It is known that projective fans, i.e., fans over the faces of a convex polytope, are shellable.
\end{remark}

There are two natural operations that take shellable fans to disjoint unions of shellable fans. Define:
\begin{align*}
\partial: \Delta &\mapsto \partial\Delta, \\
C: \Delta &\mapsto \bigsqcup_i \Delta_i,
\end{align*}
where $\Delta_i$ are as in Definition~\ref{def-shell}. We extend the operations from disjoint unions to disjoint unions of shellable fans. In particular, we can compose the operations. When applying an $n$-fold composition of the operations to an $n$-dimensional fan, the result is a disjoint union of $m$ $0$-dimensional fans. We identify this disjoint union with the integer $m$.

Note that $\partial\circ\partial (\Delta) = \emptyset$. Thus, to get interesting compositions of $\partial$ and $C$, we have to precede each $\partial$ with a $C$, unless $\partial$ is the first operation in the composition. Define a third operation:
\[ D: \Delta \mapsto \partial\circ C(\Delta).\]
Then the interesting compositions that yield non-empty results are of the form
\[ M(C,D), \quad M(C,D)\circ\partial,\]
where $M(C,D)$ is a monomial in $C$ and $D$.

The following theorem is a special case of results proved in \cite{Karu}.

\begin{theorem}\label{thm-old}
Let $\Delta$ be an $n$-dimensional shellable fan. Then the polynomial $\Psi_\Delta$ has the form
\[ \Psi_\Delta = f_n(\bc,\bd) + g_{n-1}(\bc,\bd) t_n\]
where $f_n$ and $g_{n-1}$ are homogeneous $\cd$-polynomials of degree $n$ and $n-1$, respectively. Moreover
\begin{enumerate}
\item If $M(c,d)$ is a $\cd$-monomial of degree $n$, then the coefficient of $M$ in $f_n$ is $M(C,D)(\Delta)$. 
\item If $M(c,d)$ is a $\cd$-monomial of degree $n-1$, then the coefficient of $M$ in $g_{n-1}$ is $M(C,D)\circ\partial(\Delta)$. 
 \end{enumerate}
 \end{theorem}

\begin{remark}
The theorem implies that $g_{n-1} = \Psi_{\partial \Delta}$ and that the $\cd$-polynomials $f_n$ and $g_{n-1}$ have nonnegative coefficients. The operation $C$ depends on the choice of the decomposition in the definition of shelling, but the $n$-fold composition of $\partial$ an $C$ applied to $\Delta$ is independent of the choices. 

Another implication of the theorem is that compositions of the operations $\partial$ and $C$ determine all flag numbers $f_S$ of $\Delta$. To see in a different way why this is true, note that if $\Delta$ is a fan over a sphere, then the flag numbers of $\Delta^{\leq n-1}$ determine the flag numbers of $\Delta$. If $\Delta$ is a fan over a disk, then we need the flag numbers of both $\partial\Delta$ and $\Delta^{\leq n-1}$ to determine the flag numbers of $\Delta$. Cutting $\Int \Delta$ into disjoint pieces simply expresses the set of chains in $\Int \Delta$ as a disjoint union (indexed by $\Delta_i$ containing the smallest element in the chain). Thus, the recursive construction of shelling explains why the flag numbers of $\Delta$ are determined by the operations $C$ and $\partial$.
\end{remark}

Given a monomial $M(\partial,C)$, we define its multidegree  by replacing each $C$ with $0$ and each $\partial$ with $1$. This in particular defines the multidegree of a monomial $M(C,D)$, which agrees with the previously defined multidegree of $M(\bc,\bd)$.

The following theorem together with Theorem~\ref{thm-old} implies Theorem~\ref{thm-ineq}. Note that by a Gorenstein* shellable poset we simply mean a shellable fan over a sphere.

\begin{theorem} \label{thm-ineq2}
Let $M_1, M_2, M_3$ be be $n$-fold compositions of the operations $\partial$ and $C$, such that 
\[ \mdeg M_1 + \mdeg M_2 = \mdeg M_3.\]
Then for any shellable $n$-dimensional fan $\Delta$ 
\[ M_3(\Delta) \leq M_1(\Delta) \cdot M_2(\Delta).\]
\end{theorem}

\begin{remark}
Theorem~\ref{thm-ineq2} is more general than Theorem~\ref{thm-ineq} because $\Delta$ need not be a fan over a sphere and the monomials $M_i$ may end with $\partial$. 
\end{remark}

We need the following lemma.

\begin{lemma}
Let $\Delta$ be an $n$-dimensional shellable fan and $\Pi\subset \Delta$ an $m$-dimensional shellable subfan, where  $\Pi$ is a fan over the $(m-1)$-sphere, $m<n$. If $M(\partial,C)$ is any $m$-fold composition of $\partial$ and $C$, then
\[ M(\Pi) \leq M\circ C^{n-m}(\Delta).\]
\end{lemma}

\begin{proof}
I only know how to prove this using sheaves on the fan $\Delta$. Using notation from \cite{Karu}, we replace the fans by the constant sheaves on them to get an exact sequence:
\[ 0\to \RR_\Pi \longrightarrow \RR_{\Delta^{\leq m}} \longrightarrow \RR_{\Delta^{\leq m}\setmin \Pi} \to 0.\]
The sheaves here are Cohen-Macaulay sheaves on $\Delta^{\leq m}$. The operations $C$ and $D$ can be applied to any Cohen-Macaulay sheaf and the result agrees with our definition of the operators $C$ and $D$. (Since by assumption $\partial\Pi = \emptyset$, we may assume that $M(\partial,C)$ is in fact a monomial in $C$ and $D$.) From \cite{Karu} the operations $C$ and $D$ are additive on exact sequences, hence
\[  M\circ C^{n-m}(\RR_\Delta) = M(\RR_{\Delta^{\leq m}}) = M(\RR_\Pi) + M(\RR_{\Delta^{\leq m}\setmin \Pi}).\]
Since all summands are nonnegative, the result follows.
\end{proof}

\begin{proof}[Proof of Theorem~\ref{thm-ineq}]
 
Switching $M_1$ and $M_2$ if necessary, we may assume that
\begin{align*}
M_1 &= M_1' \partial C^j,\\
M_2 &= M_2' C C^j,\\
M_3 &= M_3' \partial C^j,
\end{align*}
where $j\geq 0$ and
\[ \mdeg M_3' = \mdeg M_1' + \mdeg M_2'.\]
Since $\Delta$ is shellable, we get 
\[ \partial C^j(\Delta) = \bigsqcup_i \Pi_i,\]
where $\Pi_i$ are $(n-j-1)$-dimensional fans over spheres.  Moreover, $\Pi_i$ are subfans of $\Delta$, hence the previous lemma applies. Now
\begin{align*}
 M_3(\Delta) = \sum_i M_3'(\Pi_i) &\leq \sum_i M_1'(\Pi_i) \cdot M_2'(\Pi_i)\\
 &\leq \sum_i M_1'(\Pi_i) \cdot M_2'C^{j+1}(\Delta)\\
 &= (\sum_i M_1'(\Pi_i)) \cdot M_2'C^{j+1}(\Delta) = M_1(\Delta) \cdot M_2(\Delta).
\end{align*}
Here in the first inequality we used induction on dimension of the fans and in the second inequality we applied the lemma to the fans $\Pi_i\subset \Delta$ and monomial $M_2'$.
\end{proof}

\section{Boolean lattices}

Let $\Pi_n$ be the $n$-dimensional fan over the boundary of an $(n+1)$-simplex. The poset of cones in $\Pi_n$ is the Boolean lattice $B_{n+1}$. Geometrically, $\Pi_n$ is the fan of the toric variety $\PP^n$. 

For $\Pi_n$ we can write down an explicit shelling. Let $\sigma_n$ be the fan consisting of an $n$-dimensional simplicial cone and all its faces. In the shelling of $\Pi_n$ we will encounter fans of the type $\Pi_k \times \Pi_l$ and $\Pi_k\times \sigma_l$. These include the special cases $\Pi_k = \Pi_k\times \sigma_0$ and $\sigma_l = \Pi_0\times\sigma_l$. One can now easily verify:
\begin{align*}
 \partial (\Pi_k\times\Pi_l) &= \emptyset,\\
\partial (\Pi_k\times \sigma_l) &= \Pi_k\times \partial \sigma_l = \Pi_k\times \Pi_{l-1},\\
C (\Pi_k\times \sigma_l) &= C (\Pi_k) \times \sigma_l = \bigsqcup_{i=1,\ldots,k} \Pi_{i-1}\times \sigma_{k-i} \times \sigma_l =  \bigsqcup_{i=1,\ldots,k} \Pi_{i-1}\times \sigma_{k+l-i},\\
C (\Pi_k\times\Pi_l) &= \bigsqcup_{ {i=0,\ldots, k \atop j=0,\ldots,l} \atop (i,j) \neq (0,0)} \Pi_{i+j-1} \times \sigma_{k+l-i-j}.
\end{align*}

We want to count the number of 0-dimensional fans obtained by applying an $n$-fold composition of $\partial$ and $C$ to $\Pi_n$. Each such $0$-dimensional fan is the result of making a choice at every step we apply $C$: choose either $i$ or $(i,j)$. We encode these choices in a sequence of $n+1$ integers, one for each fan, as follows.  Applying $C$ to $\Pi_k\times \sigma_l$, the numbers are:
\[
\begin{tabu}{cccc}
  C: &  \Pi_k\times\sigma_l & \xmapsto{\text{choose $1\leq i\leq k$}} & \Pi_{i-1}\times\sigma_l \\
&k &\xmapsto{\hspace*{2.2cm}} & i-1
\end{tabu}
\]
Applying $\partial$ followed by $C$:
\[ 
\begin{tabu}{ccccc}
 \Pi_k \times \sigma_{l+1} & \stackrel{\partial}{\longmapsto} & \Pi_k\times\Pi_l &\xmapsto{\text{C: choose $i,j$}}& \Pi_{i+j-1}\times\sigma_{k+l-i-j} \\
k & \longmapsto & k+j & \xmapsto{\hspace*{1.7cm}} & i+j-1
\end{tabu}
\]
If $\partial$ is the last operation:
\[
\begin{tabu}{ccc}
 \Pi_0 \times \sigma_{1} & \stackrel{\partial}{\longmapsto} & \Pi_0\times\Pi_0  \\
0 & \longmapsto &  0 
\end{tabu}
\]

Applying a degree $n$ monomial $M(\partial,C)$ to $\Pi_n$ and fixing a choice at each step, we get a sequence of $n+1$ fans and the associated sequence of $n+1$ integers:
\[ n=i_n \longmapsto i_{n-1} \longmapsto \cdots\longmapsto i_0 = 0.\]
We define the function $f: \{0,\ldots, n\} \to \ZZ$ by $f(j) = i_j$. 

\begin{example}
Consider the monomial $\bc\bd\bd\bc = C\partial C\partial C C$ applied to $\Pi_6$:
\[
\begin{tabu}{ccccccccccccc}
\Pi_6 &\stackrel{C}{\mapsto} & \Pi_3\times \sigma_2 &\stackrel{C}{\mapsto} & \Pi_2\times \sigma_2 &\stackrel{\partial}{\mapsto} & \Pi_2\times \Pi_1 &\stackrel{C: (1,0)}{\mapsto} & \Pi_0\times \sigma_2 &\stackrel{\partial}{\mapsto} & \Pi_0\times \Pi_1 &\stackrel{C: (0,1)}{\mapsto}& \Pi_0\times \Pi_0 \\
6 &\mapsto & 3 &\mapsto & 2 &\mapsto & 2 &\mapsto & 0 &\mapsto & 1 &\mapsto &0
\end{tabu}
\]
Note that when applying $C$ to $\Pi_k\times\Pi_l$, we need to specify the choice of $(i,j)$ as this is not determined by the resulting fan.
The function constructed from this sequence:
\[ (f(0), f(1), f(2), f(3), f(4), f(5), f(6)) = (0,1,0,2,2,3,6)\]
is the admissible function for $\bc\bd\bd\bc$ shown in Figure~1(a).
\end{example}

We claim that the functions constructed as above for a $\bc\bd$-monomial $M(\bc,\bd)$ are precisely the admissible functions for $M$.  This implies that the number of admissible functions for $M$ is the coefficient of $M$ in $\Psi_{\Pi_n}$.

\begin{lemma}
 A function $f$ is admissible for a $\bc\bd$-monomial $M$ if and only if it is obtained by the construction above.  
\end{lemma}

\begin{proof}
Let us first check that a function constructed above is admissible.  First notice that the integers assigned to each fan are nonnegative and no bigger than the dimension of the fan. The first three conditions of admissibility are now easy to check from the definition. The ``bound on descent'' applies to the sequence
\[ k \longmapsto k+j \longmapsto i+j-1.\]
We need to check that 
\[ (k+j)-k \leq (i+j-1)+1,\] 
which follows from $i\geq 0$.

Conversely, given an admissible function, that means, a sequence of numbers $i_n\mapsto i_{n-1} \mapsto \cdots \mapsto i_0$, we construct the sequence of fans. Most numbers $i_j$ correspond to fans $\Pi_{i_j}\times \sigma_{j-i_j}$. The exceptions are the targets of the operation $\partial$, where the fans are
\[ 
\begin{tabu}{ccc}
i_k & \longmapsto & i_{k-1}=i_k+j  \\
 \Pi_{i_k} \times \sigma_{k-i_k} & \stackrel{\partial}{\longmapsto} & \Pi_{i_k}\times\Pi_{k-i_k-1} 
\end{tabu}
\]
If $\partial$ is followed by $C$, then from the sequence of numbers 
\[ i_k  \longmapsto  i_{k-1}=i_k+j  \longmapsto i_{k-2} = i+j -1 \]
we get the choice $(i,j) = (i_{k-2}-i_{k-1}+i_k +1,i_{k-1}-i_k)$ when applying the operation $C$. The fact that $0\leq i \leq i_k$, $0\leq j \leq k-i_k-1$ and $(i,j)\neq (0,0)$ follows from the admissibility of the function. 
\end{proof}

Recall the algebra $A$ defined in the introduction. It has a basis consisting of admissible functions for all $\bc\bd$-monomials of degree $n$, and multiplication is defined so that the product of two basis elements is again a basis element or zero.

\begin{lemma}
$A$ is a standard multigraded algebra.
\end{lemma}
 
\begin{proof}
The multiplication is clearly commutative. The function $f(i)=i$ acts as $1$. We need  to prove associativity of the multiplication. The operation of taking the minimum is associative. We only need to rule out the case that $f_1\cdot f_2=0$, but $f_1\cdot(f_2\cdot f_3)\neq 0$. In this case the supports of $f_1, f_2, f_3$ are disjoint and $g=\min(f_1,f_2)$ is not admissible, it does not satisfy the ``bound on descent'' condition. Let us show that then $\min(g,f_3)$ also does not satisfy this condition.

Suppose $g$ has a descent at $i$: $g(i) \leq g(i-1) > g(i-2)$ and the descent is too big, $g(i-1)-g(i) > g(i-2)+1$. These conditions imply that $i, i-2 \in \supp(g)$. The support condition implies that $f_3(i-2)= i-2$, and since $f_3$ must have an ascent at $i-1$, $f_3(i-2)<f_3(i-1)=i-1$. Thus, none of $i-2, i-1, i$ lie in $\supp(f_3)$, hence $\min(g,f_3)$ takes the same values as $g$ at $i-2,i-1, i$, and the ``bound on descent" condition is also violated for $\min(g,f_3)$.

It remains to prove that $A$ is generated in degrees $e_i$, $i=1,\ldots,n$. Let $f$ be an admissible function that has descents at $l_1,l_2,\ldots,l_m$, $f(l_j-1)\geq f(l_j)$. Define for $j=0,\ldots,m-1$
\[ f_j(i) = \begin{cases}
             f(i) & \text{if $i=l_j,\ldots,l_{j+1}-1$,}\\ i &\text{otherwise,}
            \end{cases}
 \]
and
\[ f_m(i) = \begin{cases}
             f(i) & \text{if $i=l_m,\ldots,n$,}\\ i &\text{otherwise.}
            \end{cases}
 \]
Then $f_j$ is an admissible function  lying in degree $e_{l_j}$ and $f=f_1\cdot f_2\cdots f_m$.
\end{proof}

This finishes the proof of Theorem~\ref{thm-simpl} for the posets $B_{n+1}$.

\section{General simplicial spheres} 

Consider now an $n$-dimensional fan $\Delta$ over a Gorenstein*  simplicial complex (note that the complex may not be homeomorphic to a sphere, it is only a homology sphere). 

First assume that $\Delta$ is shellable in the usual sense: there exists an ordering of the maximal cones $\delta_1,\delta_2,\ldots,\delta_N$, such that if $\Sigma_j$ is the fan generated by $\delta_1,\ldots,\delta_j$, then $\delta_{j+1}\cap\Sigma_j$ is a purely $n-1$-dimensional subfan of $\partial \delta_{j+1}$.  This subfan is then necessarily of the form $\Pi_k\times \sigma_{n-k-1}$ for some $k=0,\ldots,n-1$. In fact, the number of fans  $\Pi_k\times \sigma_{n-k-1}$ we get by shelling $\Delta$ is equal to the number $h_{k+1}$ of $\Delta$. Thus, the shelling of $\Delta$ differs from the shelling of $\Pi_n$ only at the first step; while $\Pi_n$ yields one copy of $\Pi_k\times \sigma_{n-k-1}$ for each $k=0,\ldots,n-1$,  a general $\Delta$ yields $h_{k+1}$ copies of the same fan. We encode this in the admissible functions $f:\{0,\ldots,n\}\to \ZZ$ as follows.  We draw $h_{k+1}$ edges from $(n-1,k)$ to $(n,n)$. Now an admissible function is a piecewise linear function $f:[0,n]\to\RR$ as before, except that if $f(n-1)=k$, then the graph of $f$ on $[n-1,n]$ must be any one of the $h_{k+1}$ edges. (See Figre 2.) The number of such admissible functions again gives the coefficients of the $\bc\bd$-index of $\Delta$ and the algebra $A$ is constructed the same way as for the fan $\Pi_n$. 

\begin{figure} [ht]
\centerline{\psfig{figure=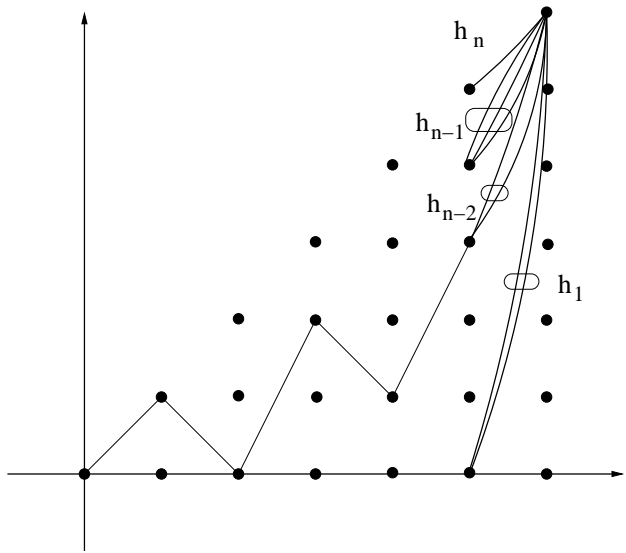,width=.5\linewidth}}
\caption{Graphs of admissible function.}
\end{figure}

Now let $\Delta$ be a general Gorenstein* simplicial fan, not necessarily shellable. Let $h_k$ be the $h$-numbers of $\Delta$ and construct the algebra $A$ using the numbers $h_k$ as in the shellable case. This algebra has the correct dimensions of graded pieces (the coefficients of the $\bc\bd$-index of $\Delta$) because the numbers $h_k$ determine the face numbers $f_i$ of $\Delta$, and these determine the flag numbers $f_S$, hence also the $\bc\bd$-index of $\Delta$. This proves Theorem~\ref{thm-simpl}.

\bibliographystyle{plain}
\bibliography{mvectbib}

\begin{thebibliography}{10}

\bibitem{BayerBillera}
Margaret~M. Bayer and Louis~J. Billera.
\newblock Generalized {D}ehn-{S}ommerville relations for polytopes, spheres and
  {E}ulerian partially ordered sets.
\newblock {\em Invent. Math.}, 79(1):143--157, 1985.

\bibitem{BayerKlapper}
Margaret~M. Bayer and Andrew Klapper.
\newblock A new index for polytopes.
\newblock {\em Discrete Comput. Geom.}, 6(1):33--47, 1991.

\bibitem{BilleraLee}
Louis~J. Billera and Carl~W. Lee.
\newblock Sufficiency of {M}c{M}ullen's conditions for {$f$}-vectors of
  simplicial polytopes.
\newblock {\em Bull. Amer. Math. Soc. (N.S.)}, 2(1):181--185, 1980.

\bibitem{FanHe}
Neil~J.Y Fan and Liao He.
\newblock The complete{$cd$}-index of boolean lattices.
\newblock {\em Preprint}, 2014.

\bibitem{Fulton}
William Fulton.
\newblock {\em Introduction to toric varieties}, volume 131 of {\em Annals of
  Mathematics Studies}.
\newblock Princeton University Press, Princeton, NJ, 1993.
\newblock The William H. Roever Lectures in Geometry.

\bibitem{Karu}
Kalle Karu.
\newblock The {$cd$}-index of fans and posets.
\newblock {\em Compos. Math.}, 142(3):701--718, 2006.

\bibitem{McMullen}
P.~McMullen.
\newblock The numbers of faces of simplicial polytopes.
\newblock {\em Israel J. Math.}, 9:559--570, 1971.

\bibitem{MuraiNevo}
Satoshi Murai and Eran Nevo.
\newblock On the cd-index and {$\gamma$}-vector of {${\rm S}^*$}-shellable
  {CW}-spheres.
\newblock {\em Math. Z.}, 271(3-4):1309--1319, 2012.

\bibitem{MuraiYanagawa}
Satoshi Murai and Kohji Yanagawa.
\newblock Squarefree {$P$}-modules and the {$\bold{cd}$}-index.
\newblock {\em Adv. Math.}, 265:241--279, 2014.

\bibitem{Purtill}
Mark Purtill.
\newblock Andr\'e permutations, lexicographic shellability and the {$cd$}-index
  of a convex polytope.
\newblock {\em Trans. Amer. Math. Soc.}, 338(1):77--104, 1993.

\bibitem{Stanley}
Richard~P. Stanley.
\newblock The number of faces of a simplicial convex polytope.
\newblock {\em Adv. in Math.}, 35(3):236--238, 1980.

\bibitem{Stanley2}
Richard~P. Stanley.
\newblock Flag {$f$}-vectors and the {$cd$}-index.
\newblock {\em Math. Z.}, 216(3):483--499, 1994.

\end{thebibliography}

 \end{document}